\documentclass[12pt]{amsart}
\usepackage{amsmath,amsthm,amssymb,cite}

\DeclareMathOperator{\RE}{Re}

\numberwithin{equation}{section}
\newtheorem{theorem}{Theorem}[section]
\newtheorem{lemma}[theorem]{Lemma}
\newtheorem{corollary}[theorem]{Corollary}
\theoremstyle{remark}

\makeatletter
\@namedef{subjclassname@2010}{%
  \textup{2010} Mathematics Subject Classification}
\makeatother

\begin{document}
\title[generalized Zalcman  conjecture ]{generalized Zalcman  conjecture for some classes of  analytic functions\boldmath}

\author[V. Ravichandran]{V. Ravichandran}

\address{Department of Mathematics, University of Delhi,
Delhi--110 007, India}
\email{vravi68@gmail.com, vravi@maths.du.ac.in}

\author[S. Verma]{Shelly Verma}

\address{Department of Mathematics, University of Delhi,
Delhi--110 007, India}
\email{jmdsv.maths@gmail.com}

\begin{abstract}
For  functions $f(z)= z+ a_2 z^2 + a_3 z^3 + \cdots$ in various subclasses of  normalized analytic functions, we consider the problem of estimating the generalized Zalcman coefficient functional $\phi(f,n,m;\lambda):=|\lambda a_n a_m -a_{n+m-1}|$.  For all real parameters $\lambda$ and $ \beta<1$, we provide the sharp upper bound of $\phi(f,n,m;\lambda)$ for functions  $f$ satisfying $\RE f'(z) > \beta$ and hence settle the open problem  of estimating $\phi(f,n,m;\lambda)$ recently proposed by  Agrawal and Sahoo  [S. Agrawal and S. K. Sahoo, On coefficient functionals associated with the Zalcman conjecture, arXiv preprint, 2016]. For all real values of $\lambda$, the estimations of $\phi(f,n,m;\lambda)$ are provided for  starlike and convex functions of order $\alpha$ $(\alpha <1)$ which   are sharp for  $\lambda  \leq 0$ or for certain positive values of $\lambda$. Moreover, for certain  positive  $\lambda$, the sharp estimation of $\phi(f,n,m;\lambda)$  is given when  $f$ is   a typically real  function or a univalent function with real coefficients or is in some subclasses of  close-to-convex functions.

\end{abstract}

\keywords{Univalent functions; starlike functions; convex functions;  typically real  functions; coefficient bounds; closed  convex hull; close-to-convex functions;  coefficient functional.}

\subjclass[2010]{30C45, 30C50, 30C80}

\thanks{The second  author is supported by a grant from  National Board for Higher Mathematics, Mumbai.}
\maketitle

\section{Introduction and  Preliminaries}
Let $\mathcal{A}$ be the class of all normalized analytic  functions  of the form $f(z)= z+ a_2 z^2 + a_3 z^3 + \cdots$ defined on the open unit disc $\mathbb{D}$. The subclass of $\mathcal{A}$ consisting of univalent functions is denoted by $\mathcal{S}$. Let $\mathcal{S}_{\mathbb{R}}$ be the class of all functions in $\mathcal{S}$ with real coefficients.  For $ \alpha <1$, we denote by $\mathcal{S}^*(\alpha)$  and $\mathcal{K}(\alpha)$,  the classes of functions $f \in \mathcal{A}$ satisfying $\RE \big(zf'(z)/f(z)\big) > \alpha$ and $\RE \big( 1+zf''(z)/f'(z)\big) > \alpha$  respectively. For $0 \leq \alpha <1$, these classes  are subclasses of $\mathcal{S}$ and  were first  introduced   by Robertson  \cite{MR1503286} in 1936. Later, for  all  $\alpha <1$, these  classes were considered  in \cite{MR2118626, MR0338337}. The classes $\mathcal{S}^*:=\mathcal{S}^*(0)$ and $\mathcal{K}:= \mathcal{K}(0)$ represent the classes of starlike and convex functions respectively. We denote the closed convex hulls  of $\mathcal{S}^*(\alpha)$ and $\mathcal{K}(\alpha)$ by $H\mathcal{S}^*(\alpha)$ and $H\mathcal{K}(\alpha)$ respectively. The class of typically real functions, denoted by $T$,  consists of all  functions in $\mathcal{A}$ which have real values on the  real axis and non-real values elsewhere. Denote by $\mathcal{P}$, the class of all analytic functions $p(z)= 1+ c_1 z+c_2 z^2 + \cdots$  defined on $\mathbb{D}$  such that  $\RE p(z)>0$. The class $\mathcal{P}_{\mathbb{R}}$ consists of all functions in  $\mathcal{P}$ with real coefficients.

 In 1916, Bieberbach conjectured the inequality $|a_n| \leq n$ for $f \in \mathcal{S}$. Since then, several attempts  were made to prove the Bieberbach conjecture  which was finally  proved by de Branges in 1985. In 1960, as an approach to prove the Bieberbach conjecture,  Lawrence Zalcman conjectured that $|a_n^2-a_{2n-1}| \leq (n-1)^2$ $(n \geq 2)$ for $f \in \mathcal{S}$.  This led to several works related to Zalcman conjecture and its generalized version $ |\lambda a_n^2-a_{2n-1}|\leq \lambda n^2-2n+1$ $(\lambda \geq 0)$ for various subclasses of $\mathcal{S}$ \cite{ MR964850, MR3284304, efraimidis2014generalized, MR3542050, MR824446, MR3467599} but the Zalcman conjecture remained open for many years for the class $\mathcal{S}$. Recently, Krushkal \cite{krushkal2014short} proved the conjecture  for the class $\mathcal{S}$ by using complex geometry of the universal Teichm\"{u}ller spaces.

 In 1999, Ma \cite{MR1694809} proposed a generalized Zalcman conjecture for $f \in \mathcal{S}$  that $$|a_n a_m-a_{n+m-1}| \leq (n-1)(m-1)$$ which is still an open problem, however he proved it for the classes $\mathcal{S}^*$ and $\mathcal{S}_{\mathbb{R}}$. For $\lambda \in \mathbb{R}$, let  $\phi(f,n,m;\lambda):= |\lambda a_n a_m-a_{n+m-1}|$ denote the generalized Zalcman coefficient functional over $\mathcal{A}$. For  $\beta<1$, the class $\mathcal{C}(\beta)$ of close-to-convex functions of order $\beta$ consists of $f \in  \mathcal{A}$ such that $\RE \big(z f'(z)/ \big( e^{i\theta} g(z)\big)\big) > \beta$  for some $g \in \mathcal{S}^*$ and  $\theta \in \mathbb{R}$. For $0 \leq \beta < 1$, the class $\mathcal{C}(\beta)$  is a subclass of $\mathcal{S}$ and was considered in \cite{MR0160890} in a more general form.  The class of close-to-convex functions is denoted by $\mathcal{C}:=\mathcal{C}(0)$, for details, see  \cite{MR704184}.  Let $\mathcal{F}_1(\beta)$ and $\mathcal{F}_2(\beta)$ be the subclasses of $\mathcal{C}(\beta)$ $(\beta<1)$ corresponding to  $\theta=0$ and the starlike functions $g(z)=z/(1-z)$ and $g(z)=z/(1-z^2)$ respectively. For $\beta <1$, let $\mathcal{R}(\beta)$ denote the class of functions $f \in \mathcal{A}$ satisfying $\RE f'(z) > \beta$.   For $0 \leq \beta <1$, $\mathcal{R}(\beta)$ is a subclass of $\mathcal{S}$ and  was  first introduced in  \cite{MR0338338}. Here, we are interested in $\mathcal{R}(\beta)$ for all values of $\beta$ $(\beta <1)$. Recently,  for some  positive values of $\lambda$ and  $0 \leq \beta <1$, Agrawal and Sahoo \cite{agrawal2016coefficient} gave the sharp estimation of $\phi(f,n,m;\lambda)$ for the classes $\mathcal{R}(\beta)$ and $H\mathcal{K}$.

 In this paper, for all real values of  $\lambda$, we give  the sharp estimation of  $\phi(f,n,m;\lambda)$ for $f \in \mathcal{R}(\beta)$ $(\beta<1)$. Also, for $f \in \mathcal{S}^*(\alpha)$ and $f \in \mathcal{K}(\alpha)$ $( \alpha <1)$, the estimations of $\phi(f,n,m;\lambda)$  are given for all real values of $\lambda$ which are sharp  when $\lambda \leq 0$ or when $\lambda$ is taking certain  positive values. Moreover, for certain positive values of $\lambda$, the sharp estimations of $\phi(f,n,m;\lambda)$ are provided for the classes $T$, $\mathcal{S}_{\mathbb{R}}$,  $\mathcal{F}_1(\beta)$ and $\mathcal{F}_2(\beta)$ $(\beta <1)$.

 We prove our results either by applying the  well-known estimation of  $|\lambda c_n c_m - c_{n+m}|$   for  $p(z)=1+\sum_{n=1}^{\infty} c_n z^n \in \mathcal{P}$  or by applying some characterization of functions in the class $\mathcal{P}$ and that of typically real functions in terms of some positive semi-definite Hermitian form, see \cite{MR760980, MR3348983}.  Earlier,  such  characterization of functions with positive real part in terms of some  positive semi-definite Hermitian form \cite{MR760980} was used in \cite{MR1387562,MR2055766,MR3348983}. It should be pointed out that  in the literature, for various subclasses of $\mathcal{S}$ which are invariant under rotations, the estimation of $\phi(f,n,n;\lambda)$ is usually obtained by using the fact that the expression $\phi(f,n,n;\lambda)$ is invariant under rotations and by an application of the Cauchy-Schwarz inequality which requires $\lambda$ to be non-negative. However,  we are able to give the sharp estimation of $\phi(f,n,m;\lambda)$   for various subclasses of $\mathcal{A}$ when $\lambda \leq 0$. Moreover, for certain positive  $\lambda$, our technique is giving the estimation of $\phi(f,n,m;\lambda)$  when $f$  is in some subclasses of $\mathcal{A}$  which are not necessarily invariant under rotations.  We need the following lemmas to prove our results.
\begin{lemma}\cite[Lemma 2.3,\ p.\ 507]{MR3348983} \label{p4lem1}If $p(z)= 1 + \sum_{k=1}^{\infty} c_k z^k \in \mathcal{P}$, then for all $n, m \in \mathbb{N}$, $$|\mu c_n c_m - c_{n+m}| \leq \begin{cases} 2, &\text{$ 0 \leq \mu \leq 1$;}\\
2|2 \mu -1|, &\text{elsewhere.}
\end{cases}  $$
The result is sharp.
\end{lemma}
\begin{lemma}\cite[Theorem 4(b), p.\ 678]{MR760980} \label{p4lem2}
A function $p(z)= 1 + \sum_{k=1}^{\infty} c_k z^k \in \mathcal{P}$ if and only if
\begin{align}
\sum_{j=0}^{\infty} \left\lbrace \left|2 z_j + \sum_{k=1}^{\infty} c_k z_{k+j} \right|^2 - \left|\sum_{k=0}^{\infty} c_{k+1} z_{k+j} \right|^2 \right\rbrace &\geq 0 \notag
\end{align}
for every sequence $\{z_k\}$ of complex numbers which satisfy $\limsup_{k \to \infty}  |z_k|^{1/k} < 1$.
\end{lemma}

\begin{lemma}\cite[Theorem 4(f), p.\ 678]{MR760980}\label{p4lem2.1}
A function $f(z)= z + \sum_{k=2}^{\infty} a_k z^k \in T$ if and only if
\begin{align}
\sum_{j=0}^{\infty} \left\lbrace \left|2 z_j + \sum_{k=1}^{\infty}(a_{k+1}- a_{k-1}) z_{k+j} \right|^2 - \left|\sum_{k=0}^{\infty} (a_{k+2}-a_k) z_{k+j} \right|^2 \right\rbrace &\geq 0 \notag
\end{align}
for every sequence $\{z_k\}$ of complex numbers which satisfy $\limsup_{k \to \infty}  |z_k|^{1/k} < 1$.
\end{lemma}

\begin{lemma}\label{p4lem3}
Let $\nu(t)$ be a probability measure on $[0,2\pi]$. Then  for all $n,m \in \mathbb{N}$,
\begin{align*}
\left|\lambda \int_{0}^{2 \pi} e^{i n t} \, d \nu(t) \int_{0}^{2 \pi} e^{i m t} \, d \nu(t) - \int_{0}^{2 \pi} e^{i (n+m) t} \, d \nu(t)\right| \leq \begin{cases} 1, &\text{$ 0 \leq \lambda \leq 2$;}\\
| \lambda -1|, &\text{elsewhere.}
\end{cases}
\end{align*}
\end{lemma}
\begin{proof}
The function  $p(z)= 1+ \sum_{n=1}^{\infty} c_n z^n $ given  by the Herglotz representation formula \cite[Corollary 3.6, p.\ 30]{MR768747},
$$p(z)= \int_0^{2\pi} \dfrac{1+ e^{i t}z}{1- e^{i t}z} \, d \nu(t)$$ is clearly in $\mathcal{P}$. On comparing the coefficients on both sides in the above equation, we obtain
$$c_n=2\int_0^{2\pi}  e^{i n t} \, d \nu(t) \quad (n \geq 1).$$
An  application of Lemma \ref{p4lem1}  to the function $p$ gives
\begin{align*}
\left|2 \mu \int_{0}^{2 \pi} e^{i n t} \, d \nu(t) \int_{0}^{2 \pi} e^{i m t} \, d \nu(t) - \int_{0}^{2 \pi} e^{i (n+m) t} \, d \nu(t)\right| \leq \begin{cases} 1, &\text{$ 0 \leq \mu \leq 1$;}\\
| 2\mu -1|, &\text{elsewhere.}
\end{cases}
\end{align*} On substituting $\lambda=2 \mu$, the desired estimates follow.
\end{proof}
For $\lambda=2$, the above lemma is proved in \cite[Lemma 2.1, p.\ 330]{MR1694809}.
\section{Generalized Zalcman conjecture for $\mathcal{S}^*(\alpha)$ and $\mathcal{K}(\alpha)$}

For $ \alpha <1$,  define  a function $f_1: \mathbb{D} \to \mathbb{C}$ by
\begin{align}
f_1(z):= \dfrac{z}{(1- z)^{2(1-\alpha)}}= z+  \sum_{n=2}^{\infty} A_n z^n, \label{p4eqllll}
\end{align}
where
\begin{align}
A_n= \dfrac{1}{(n-1)!}\prod_{j=0}^{n-2} \big(2(1-\alpha)+j\big). \label{p4eq1}
\end{align}
It is known that $f_1$  and its rotations work as  extremal functions for the coefficient bounds of functions in the class $\mathcal{S}^*(\alpha)$ \cite[Theorem 5.6, p.\ 324]{MR2118626}.  Therefore, they  could be the expected extremal functions  for the  upper bound  of the  generalized Zalcman coefficient functional $\phi(f,n,m;\lambda)$ when $f \in \mathcal{S}^*(\alpha) $. This is shown to be true by the following theorem at least when   $\lambda \geq 2 A_{n+m-1}/(A_n A_m)$ or $\lambda  \leq 0$.

\begin{theorem}\label{p4thm1}
If $f(z)= z + \sum_{n=2}^{\infty}a_n z^n \in H\mathcal{S}^*(\alpha)$ $(\alpha <1)$, then  for all $n,m=2, 3, \ldots$,
\begin{align*}
\left|\lambda a_n a_m - a_{n+m-1}\right| \leq \begin{cases}A_{n+m-1}, &\text{$ 0 \leq \lambda \leq \dfrac{2 A_{n+m-1}}{A_n A_m}$;}\\
| \lambda A_n A_m -A_{n+m-1}|, &\text{elsewhere,}
\end{cases}
\end{align*}
where $A_n$ is given by  \eqref{p4eq1}.
The second inequality is sharp for the function  $f_1$  and its rotations where $f_1$ is given  by the equation \eqref{p4eqllll}.
\end{theorem}

\begin{proof}
Since $f \in H\mathcal{S}^*(\alpha)$ $(\alpha <1)$,  there exists a probability measure $\nu(t)$ on $[0,2\pi]$  \cite[Theorem 3, p.\ 417 ]{MR0338337}  such that $$f(z)=\int_{0}^{2 \pi} \dfrac{z}{(1- e^{it}z)^{2(1-\alpha)}} \, d \nu(t).$$On comparing the coefficients on both sides, we obtain
$$a_n= A_n \int_0^{2\pi}  e^{i (n-1) t} \, d \nu(t) \quad (n \geq 2),$$ where $A_n$ is given by the equation  \eqref{p4eq1}. This implies
\begin{align*}
&\left|\lambda a_n a_m - a_{n+m-1}\right| \\ &{}= A_{n+m-1}\left|\lambda\dfrac{A_n A_m}{A_{n+m-1}} \int_{0}^{2 \pi} e^{i (n-1) t} \, d \nu(t) \int_{0}^{2 \pi} e^{i (m-1) t} \, d \nu(t) - \int_{0}^{2 \pi} e^{i (n+m-2) t} \, d \nu(t)\right|.
\end{align*}
An application of  Lemma \ref{p4lem3}  to the above equation  yields
\begin{align*}
&\left|\lambda a_n a_m - a_{n+m-1}\right| \leq \begin{cases}A_{n+m-1}, &\text{$ 0 \leq \lambda \leq \dfrac{2 A_{n+m-1}}{A_n A_m}$;}\\
| \lambda A_n A_m -A_{n+m-1}|, &\text{elsewhere.}
\end{cases}  \qedhere
\end{align*}
\end{proof}

For $m=n$, we have the following sharp result.
\begin{corollary}
If $f(z)= z + \sum_{n=2}^{\infty}a_n z^n \in H\mathcal{S}^*(\alpha)$ $(\alpha<1)$, then for all $n=2, 3, \ldots$,
\begin{align*}
\left|\lambda a_n^2  - a_{2n-1}\right| \leq \begin{cases}A_{2n-1}, &\text{$ 0 \leq \lambda \leq \dfrac{2 A_{2n-1}}{A_n^2}$;}\\
| \lambda A_n^2 -A_{2n-1}|, &\text{elsewhere,}
\end{cases}
\end{align*}
where $A_n$ is given by  \eqref{p4eq1}. The second inequality is sharp for the function $f_1$, given by  the equation  \eqref{p4eqllll},  and its rotations  whereas the first  inequality is sharp for the function of the form
\begin{align}
f(z)= \sum_{k=1}^{2(n-1)} m_k g_k(z), \label{p4eq111.1}
\end{align}
where $0 \leq m_k \leq 1$,  $\sum_{k=1}^{n-1} m_{2k}=\sum_{k=1}^{n-1} m_{2k-1}=1/2$,  $g_k(z)= e^{-i \theta_k} f_1(e^{i \theta_k}z)$ and  $\theta_k= (2k+1)\pi/(2n-2)$.
\end{corollary}

For $\alpha=0$ and $\lambda \geq 0$, the above corollary reduces to the inequalities mentioned in \cite[p.\ 474]{MR824446}. It is a well-known result given by Alexander that a function $f \in \mathcal{A}$ is  in  $\mathcal{K}$ if and only if $z f'(z) \in \mathcal{S}^*$. This implies that for $ \alpha <1$,  $f \in H\mathcal{K}(\alpha)$ if and only if $z f'(z) \in H\mathcal{S}^*(\alpha)$ and therefore, we have the following deduction from  the Theorem \ref{p4thm1}.

\begin{corollary}\label{p4cor1}
If $f(z)= z + \sum_{n=2}^{\infty}a_n z^n \in H\mathcal{K}(\alpha)$ $(\alpha<1)$, then for all $n,m=2, 3, \ldots$,
\begin{align*}
\left|\lambda a_n a_m - a_{n+m-1}\right| \leq \begin{cases}\dfrac{A_{n+m-1}}{n+m-1}, &\text{$ 0 \leq \lambda \leq \dfrac{2 n m A_{n+m-1}}{(n+m-1)A_n A_m}$;}\\
\left| \lambda \dfrac{ A_n A_m}{n m } - \dfrac{A_{n+m-1}}{n+m-1}\right|, &\text{elsewhere,}
\end{cases}
\end{align*}
where $A_n$ is given by the equation  \eqref{p4eq1}.  The second inequality is sharp for the function  $f_2$ and its rotations, where
\begin{align}
f_2(z)= \begin{cases}
\dfrac{(1-z)^{-(1-2 \alpha)}-1}{1-2\alpha}, &\text{ $\alpha \neq 1/2$;}\\
 -\log{(1-z)}, &\text{ $\alpha = 1/2$.}
\end{cases} \label{p4eq111.2}
\end{align}
\end{corollary}
For $\alpha=0$ and $\lambda \geq 2$, the above  corollary  reduces  to  \cite[Theorem 2.1, p.\ 3]{agrawal2016coefficient}.
For $m=n$, we have the following sharp result which has been proved in \cite{MR3542050} by maximizing the real-valued functional $\RE(\lambda a_n^2-a_{2n-1})$ for the case $\lambda \geq 0$.
\begin{corollary}\label{p4corr1}
If $f(z)= z + \sum_{n=2}^{\infty}a_n z^n \in H\mathcal{K}(\alpha)$ $(\alpha<1)$, then for all $n=2, 3, \ldots$,
\begin{align*}
\left|\lambda a_n^2  - a_{2n-1}\right| \leq \begin{cases}\dfrac{A_{2n-1}}{2n-1}, &\text{$ 0 \leq \lambda \leq \dfrac{2 n^2 A_{2n-1}}{(2n-1)A_n^2 }$;}\\
\left| \lambda \dfrac{ A_n^2 }{n^2} - \dfrac{A_{2n-1}}{2n-1}\right|, &\text{elsewhere,}
\end{cases}
\end{align*}
where $A_n$ is given by the equation  \eqref{p4eq1}. The second inequality is sharp for the function $f_2$, given by \eqref{p4eq111.2}, and its rotations whereas the first inequality is sharp for the function given by  the equation  \eqref{p4eq111.1} with  $g_k(z)= e^{-i \theta_k} f_2(e^{i \theta_k} z)$.

\end{corollary}
If $\lambda \geq 0$, the above corollary reduces to \cite[Theorem 3.3]{MR3467599} and \cite[Theorem 4]{MR3542050} for $\alpha=-1/2$ and  $\alpha=1/2$ respectively. Also, for  $\alpha=0$ and $0 \leq \lambda \leq 2$,  the above corollary  was proved in   \cite[Theorem 3, p.\ 3]{efraimidis2014generalized}.
\section{Generalized Zalcman conjecture for the class  $\mathcal{R}(\beta)$  and for typically real functions}
For $\lambda \geq nm/((1-\beta)(n+m-1))$ and $0 \leq \beta <1$, the second  inequality of the following theorem has been  recently proved by Agrawal and Sahoo \cite{agrawal2016coefficient} and they proposed it as an open problem for $ 0 < \lambda < nm/((1-\beta)(n+m-1))$ which has now  been  settled  in the following theorem by making use of the Hermitian form for functions in the class $\mathcal{P}$.
\begin{theorem}\label{p4thm2}
If $f(z)=z+\sum_{n=2}^{\infty}a_n z^n \in \mathcal{R}(\beta)$ $(\beta <1)$, then for all $n,m=2, 3, \ldots$,
\begin{align*}
\left|\lambda a_n a_m - a_{n+m-1}\right| \leq \begin{cases}\dfrac{2(1-\beta)}{n+m-1}, &\text{$ 0 \leq \lambda \leq \dfrac{n m }{(1-\beta)(n+m-1)}$;}\\
\left|\dfrac{4 \lambda (1-\beta)^2}{n m } - \dfrac{2(1-\beta)}{n+m-1}\right|, &\text{elsewhere.}
\end{cases}
\end{align*}
The result is sharp.
\end{theorem}
\begin{proof}
Since $f \in \mathcal{R}(\beta)$, $\big(f'(z)-\beta\big)/(1-\beta)= 1+ \sum_{n=1}^{\infty} (n+1) a_{n+1}/(1-\beta) z^n \in \mathcal{P}$ which gives
\begin{align}
|a_n| \leq \dfrac{2(1-\beta)}{n}\quad  (n \geq 2). \label{p4eqll}
\end{align}
Clearly, the bounds are sharp for the function  $f_0: \mathbb{D} \to \mathbb{C}$ defined by
\begin{align}
f_0(z)=(1-\beta)\int_0^z \dfrac{1+t}{1-t} \, dt + \beta z. \label{p4eqn}
\end{align}
For  fixed $n,m =2,3,\ldots$, choose the sequence $\{z_k\}$ of complex numbers  by  $z_{n-2}= \lambda (1-\beta)a_m$, $ z_{n+m-3}=-n(1-\beta)/(n+m-1)$, $z_k = 0$ for all $ k \neq n-2, n+m-3$. An application of  Lemma \ref{p4lem2} to the function $(f'-\beta)/(1-\beta) \in \mathcal{P}$  gives
\begin{align*}
&n^2\left|\lambda a_n a_m - a_{n+m-1}\right|^2 \\
&\leq \left|\left(2 \lambda (1-\beta) - \dfrac{mn}{n+m-1} \right) a_m \right|^2 - \left|\dfrac{mn}{n+m-1} a_m \right|^2 + \dfrac{4n^2 (1-\beta)^2}{(n+m-1)^2}\\
&= 4\lambda(1-\beta)\left(\lambda (1-\beta)-\dfrac{mn}{n+m-1}\right)|a_m|^2 + \dfrac{4n^2 (1-\beta)^2}{(n+m-1)^2}.
\end{align*}
By using  the bounds  given by \eqref{p4eqll} in the above inequality, we have
\begin{align*}
\left|\lambda a_n a_m - a_{n+m-1}\right|^2
&\leq
\begin{cases} \dfrac{4 (1-\beta)^2}{(n+m-1)^2}, &\text{$0 \leq \lambda \leq \dfrac{n m }{(1-\beta)(n+m-1)} $;}\\
 \left( \dfrac{4 \lambda(1-\beta)^2}{n m } - \dfrac{2(1-\beta)}{n+m-1}\right)^2, &\text{elsewhere.}
\end{cases}
\end{align*}
For $0 \leq \lambda \leq n m/\big((1-\beta)(n+m-1)\big)$, the inequality is sharp for the function $f(z)= (1-\beta)\int_0^z(1+t^{n+m-2})/(1-t^{n+m-2}) \, dt + \beta z$.

For $\lambda\leq 0$ or $\lambda \geq n m/\big((1-\beta)(n+m-1)\big)$, the inequality is sharp for the function $f_0$ given by the equation \eqref{p4eqn}.
\end{proof}
For $\beta=0$ and $0 < \lambda \leq 4/3$, the above theorem was proved in \cite{efraimidis2014generalized} by maximizing the real valued functional $\RE (\lambda a_n^2 - a_{2n-1})$ over $\mathcal{R}(\beta)$. Also, we have the following simple result.
\begin{corollary}
If the analytic  function $f(z)=z+\sum_{n=2}^{\infty}a_n z^n$ satisfies $\RE{\left(f(z)/z\right)} > \beta$ $(\beta <1)$ in $\mathbb{D}$, then for all $n,m=2, 3, \ldots$,
\begin{align*}
\left|\lambda a_n a_m - a_{n+m-1}\right| \leq \begin{cases} 2(1-\beta), &\text{$ 0 \leq \lambda \leq \dfrac{1}{(1-\beta)}$;}\\
2(1-\beta)\left|2\lambda (1-\beta) - 1\right|, &\text{elsewhere.}
\end{cases}
\end{align*}
The result is sharp.
\end{corollary}

The following theorem   generalizes \cite[Theorem 3.1, p.\ 335]{MR1694809} which was proved for $\lambda=1$  by induction on $n$ and $m$. Although, it can  be proved  by induction  on $n$ and $m$ but here, we are giving it as an application of the Hermitian form for typically real functions.
\begin{theorem}\label{p4thm3}
If $f(z)=z+\sum_{n=2}^{\infty}a_n z^n \in T$ and $\lambda \geq 1$, then
\begin{itemize}
\item[$(i)$] if $n=2$ and $m$ is even,   the upper bound of $|\lambda a_n a_m - a_{n+m-1}|$ is
\begin{itemize}
\item[$(a)$] $3+(2 \lambda-1)(m-2)$ for $1 \leq \lambda \leq 3/2$,
\item[$(b)$] $2 \lambda m-m-1$ for $\lambda \geq 3/2$;
\end{itemize}
\item[$(ii)$] if $m=2$ and $n$ is even, the upper bound of $|\lambda a_n a_m - a_{n+m-1}|$ is
\begin{itemize}
\item[$(a)$] $3+(2 \lambda-1)(n-2)$ for $1 \leq \lambda \leq 3/2$,
\item[$(b)$] $ 2 \lambda n-n-1$ for $\lambda \geq 3/2$;
\end{itemize}
\item[$(iii)$] in the other cases,  we have
\begin{align*}
 |\lambda a_n a_m - a_{n+m-1}|  \leq  \lambda m n -n -m +1.
\end{align*}
\end{itemize}
The bounds given by  (i)(b), (ii)(b) and (iii) are sharp whereas the bounds in (i)(a) and (ii)(a) are sharp for  $\lambda=1$ or the case when $n=2$ and  $m=2$.

\end{theorem}

\begin{proof}
For  fixed  $n, m=2,3,\ldots$, choose the sequence $\{z_k\}$ of real numbers  by  $z_{n-2}= \lambda a_m$, $ z_{n+m-3}=-1$, $z_k = 0$ for all $ k \neq n-2, n+m-3$. Since $f \in T$, $|a_n| \leq n$ $(n \geq 2)$. So,  by using Lemma \ref{p4lem2.1} to the function $f \in T$, we have
\begin{align}
\left|(\lambda a_n a_m - a_{n+m-1})-(\lambda a_{n-2} a_m - a_{n+m-3})\right|^2
&\leq |(2 \lambda- 1)a_m + a_{m-2}|^2 -|a_m -a_{m-2}|^2 + 4 \notag \\
&= 4 \lambda(\lambda-1) a_m^2 + 4 \lambda a_m a_{m-2} +4 \notag\\
&\leq 4(\lambda m -1)^2.  \label{p4eq1.1}
\end{align}
 Since $f \in T$, therefore $f(z)=\big(z/(1-z^2)\big)p(z)$ for some $p(z)= 1 + \sum_{n=1}^{\infty} c_n z^n \in \mathcal{P}_{\mathbb{R}}$.  This gives
\begin{align}
a_{2k}= c_1 + c_3 + \cdots + c_{2k-1} \quad \text{and}\quad
a_{2k+1}= 1 + c_2 + c_4 + \cdots + c_{2k}. \label{p4eq2}
\end{align}
By \cite[Theorem 1, p.\ 468]{MR824446}, we have $\lambda a_2^2 - a_3 \leq 4 \lambda-3$. Clearly, $\lambda a_2^2 - a_3 \geq -a_3 \geq -3$.  Also, we observe that  $1 \leq \lambda \leq  3/2$ is equivalent to  $1 \leq 4 \lambda -3 \leq 3$. Therefore, we have
\begin{align}
| \lambda a_2^2 - a_3| \leq
\begin{cases}
3, &\text{for $1 \leq \lambda \leq  3/2,$}\\
4 \lambda-3, &\text{for $\lambda \geq 3/2.$} \label{p4eq3.10}
\end{cases}
\end{align}
The first inequality  in  \eqref{p4eq3.10} is sharp for the function $f(z)=z(1+z^2)/(1-z^2)^2$ and the second inequality holds for the Koebe  function $k(z)=z/(1-z)^2$.
If $n=2$ and $m=2k$ $(k \geq 2)$, then
\begin{align}
|\lambda a_2 a_{m} - a_{m+1}|&= |\lambda a_2 a_{2k} - a_{2k+1}|\notag\\
&= |\lambda c_1 (c_1 + c_3 + \cdots + c_{2k-1}) - (1 + c_2 + c_4 + \cdots + c_{2k})|\notag\\
&= |(\lambda c_1^2 -c_2-1) + (\lambda c_1 c_3 -c_4) + \cdots + (\lambda c_1 c_{2k-1} -c_{2k})|\notag \\
&= |(\lambda a_2^2 -a_3) + (\lambda c_1 c_3 -c_4) + \cdots + (\lambda c_1 c_{2k-1} -c_{2k})|.\label{p4eq3.11}
\end{align}
An application  of Lemma \ref{p4lem1}  and the inequality  \eqref{p4eq3.10} in the equation \eqref{p4eq3.11} gives
\begin{align}
|\lambda a_2 a_{m} - a_{m+1}| \leq
\begin{cases}
3+(2 \lambda-1)(m-2), &\text{for $1 \leq \lambda \leq  3/2;$}\\
 2 \lambda m-m -1, &\text{for $\lambda \geq 3/2.$} \notag
\end{cases}
\end{align}
This proves $(i)$.

When $m=2$   and $n$ is even, the desired bounds in $(ii)$ follow  by interchanging the roles of $n$ and $m$ in the equation \eqref{p4eq3.11} and in the  above inequality. For $\lambda=1$, the sharpness in $(i)(a)$ and $(ii)(a)$ follow for the function $f(z)=z(1+z^2)/(1-z^2)^2$. Now, it is left to prove the inequality in the case $(iii)$.
Since $\lambda a_4^2 - a_7  \leq 16 \lambda-7$ \cite[Theorem 1, p.\ 468]{MR824446} and clearly $\lambda a_4^2 - a_7  \geq -7 \geq -9 \geq -(16 \lambda-7)$, we have
\begin{align}
|\lambda a_4^2 - a_7|  \leq 16 \lambda-7. \label{p4eq3.12}
\end{align}
  For $n=4 $ and  $m=2k$ $(k \geq 3)$, by proceeding as in the  equation \eqref{p4eq3.11},  we have
\begin{align}
&|\lambda a_4 a_{m} - a_{m+3}| \notag \\
&= |\lambda a_4 a_{2k} - a_{2k+3}|\notag\\
&= |\lambda (c_1+ c_3) (c_1 + c_3 + \cdots + c_{2k-1}) - (1 + c_2 + c_4 + \cdots + c_{2(k+1)})|\notag\\
&= |(\lambda a_4^2 -a_7)+ \lambda c_1 (c_5 +  \cdots + c_{2k-1}) + (\lambda c_3 c_5 -c_8) + \cdots + (\lambda c_3 c_{2k-1} -c_{2k+2})|.\notag
\end{align}
An application  of Lemma \ref{p4lem1}  and  the inequality  \eqref{p4eq3.12} in the above equation  gives
\begin{align}
|\lambda a_4 a_{m} - a_{m+3}| &\leq  4  \lambda m -m -3. \label{p4eq3.01}
\end{align}
Therefore, if  $n$,  $m$ are even and  $n >4$, $m>2$,  then
\begin{align}
|\lambda a_n a_{m} - a_{n+m-1}| &\leq  |(\lambda a_n a_m - a_{n+m-1})-(\lambda a_{n-2} a_m - a_{n+m-3})| \notag  \\
&\quad{}+ |(\lambda a_{n-2} a_m - a_{n+m-3})-(\lambda a_{n-4} a_m - a_{n+m-5})|  + \cdots \notag \\
&\quad{}+  |(\lambda a_{6} a_m - a_{m+5})-(\lambda a_{4} a_m - a_{m+3})| + |\lambda a_{4} a_m - a_{m+3}|.\label{p4eq3.011}
\end{align}
In view of \eqref{p4eq1.1}, \eqref{p4eq3.12}  and \eqref{p4eq3.01}, we have
\begin{align*}
|\lambda a_n a_{m} - a_{n+m-1}| \leq   (\lambda m -1)(n-4) + 4 m\lambda- m -3 = \lambda m n -m -n +1.
\end{align*}
Next, we consider the case when  $n$ is even  and  $m$ is  odd. If $n=2$ and $m=2k+1$ $(k \geq 1)$, then by proceeding similarly as in the equation   \eqref{p4eq3.11} and applying Lemma \ref{p4lem1}, we obtain
\begin{align}
|\lambda a_2 a_{m} - a_{m+1}|&= |\lambda a_2 a_{2k+1} - a_{2k+2}|\notag\\
&= |\lambda c_1 (1 + c_2 + c_4 + \cdots + c_{2k}) - (c_1 + c_3 + \cdots + c_{2k+1})|\notag\\
&\leq m(2 \lambda  -1 )-1.\label{p4eq4}
\end{align}
If  $n=2k$ $(k > 1)$ and  $m$ is odd, then by proceeding as in the inequality \eqref{p4eq3.011} and applying  \eqref{p4eq1.1} and \eqref{p4eq4}, we have
\begin{align*}
|\lambda a_n a_{m} - a_{n+m-1}| \leq  2(\lambda m -1)(k-1) + m(2 \lambda-1)-1 = \lambda m n -m -n +1.
\end{align*}
Finally, we  consider the case when  $n$ is odd. In this case, we have
\begin{align}
|\lambda a_n a_{m} - a_{n+m-1}| &\leq  |(\lambda a_n a_m - a_{n+m-1})-(\lambda a_{n-2} a_m - a_{n+m-3})| \notag  \\
&\quad{}+ |(\lambda a_{n-2} a_m - a_{n+m-3})-(\lambda a_{n-4} a_m - a_{n+m-5})|  + \cdots \notag \\
&\quad{}+  |(\lambda a_{3} a_m - a_{m+2})-(\lambda a_{1} a_m - a_{m})| + |\lambda a_{1} a_m - a_{m}| \quad (a_1=1).\notag
\end{align}
Using inequality \eqref{p4eq1.1} and  the bound of $|a_m|$ in the above inequality, we obtain
\begin{align*}
|\lambda a_n a_{m} - a_{n+m-1}| \leq  \lambda m n- m -n +1.
\end{align*}
The  sharpness in the cases $(i)(b)$, $(ii)(b)$ and $(iii)$ follow for the Koebe  function $k(z)=z/(1-z)^2$.
\end{proof}
For $\lambda=1$, the following result is given in \cite[Theorem 3.2, p.\ 338]{MR1694809}.
\begin{corollary}
If $f(z)=z+\sum_{n=2}^{\infty}a_n z^n \in \mathcal{S}_{\mathbb{R}}$ and $\lambda \geq 1$, then for $n,m=2,3,\ldots$, $$|\lambda a_n a_m - a_{n+m-1}|  \leq \lambda m n -n -m +1.$$ The result is sharp.
\end{corollary}
\begin{proof}
Since  $\mathcal{S}_{\mathbb{R}} \subset  \mathcal{S}$, by  using \cite[Theorem 2, p.\ 35]{MR704183},    we have
\begin{align}
|a_2^2 - a_{3}| \leq 1. \label{p4eqs}
\end{align}
Also,  $\mathcal{S}_{\mathbb{R}} \subset T$, therefore for $\lambda \geq 1$,  by  \cite[Theorem 1, p.\ 468]{MR824446},  we have $\lambda a_2^2-a_3 \leq 4 \lambda -3$. For $1 \leq \lambda \leq 3/2$,  an application of the inequality  \eqref{p4eqs} gives  $\lambda a_2^2-a_3 \geq a_2^2-a_3 \geq -1 \geq -(4 \lambda -3)$. Thus, in view of  the inequality \eqref{p4eq3.10}, we must have the sharp inequality
 \begin{align}
|\lambda a_2^2 - a_3| \leq 4 \lambda -3 \label{p4eql}
\end{align}
where the sharpness follows for the Koebe function $k(z)=z/(1-z)^2$.
For even $m>2$, an application  of  \eqref{p4eql} and  Lemma \ref{p4lem1}  in the equation \eqref{p4eq3.11}  gives
$$ |\lambda a_2 a_{m} - a_{m+1}| \leq   2 m\lambda-m -1.$$ When $m=2$ and $n>2$ is even, the desired estimate follows by interchanging the roles of $m$ and $n$ in the above inequality. The  other cases follow immediately from the Theorem \ref{p4thm3}.  The result is sharp for the Koebe function. \qedhere

\end{proof}
\section{Generalized Zalcman conjecture for some subclasses of  close-to-convex  functions}
 Recall that the classes $\mathcal{F}_1(\beta)$ and $\mathcal{F}_2(\beta)$ $(\beta<1)$ are defined as follows:
\begin{align*}
\mathcal{F}_1(\beta)&:=\{f \in \mathcal{A}: \RE\big((1-z)f'(z)\big)> \beta\} \intertext{and}
\mathcal{F}_2(\beta)&:=\{f \in \mathcal{A}: \RE\big((1-z^2)f'(z)\big)> \beta\}.
\end{align*}
For $0 \leq \beta <1$, the classes $\mathcal{F}_1(\beta)$ and $\mathcal{F}_2(\beta)$ are subclasses of $\mathcal{C}$, the class  of close-to-convex functions.
 Define the  functions   $f_{1,\beta}: \mathbb{D} \to \mathbb{C}$ and $f_{2,\beta}: \mathbb{D} \to \mathbb{C}$,  in $\mathcal{F}_1(\beta)$ and $\mathcal{F}_2(\beta)$ respectively, by
\begin{align}
f_{1,\beta}(z)&= \dfrac{2(1-\beta)z}{1-z}+ (1-2 \beta) \log{(1-z)} \label{p4eq4.102}
\intertext{and}
f_{2,\beta}(z)&= \dfrac{z(1-\beta ) }{1-z^2}+\dfrac{\beta  }{2} \log{\left(\dfrac{1+z}{1-z}\right)}.\notag
\end{align}
Recently,  for certain positive values of  $\lambda $, the sharp estimation of $\phi(f,n,n;\lambda)$  over $\mathcal{C}$ is given in  \cite{li2016generalized} by using the fact that $\mathcal{C}$ and $\phi(f,n,n;\lambda)$ are invariant under rotations.
Note that  the classes $\mathcal{F}_1(\beta)$ and  $ \mathcal{F}_2(\beta)$ are  not necessarily invariant under rotations. For instance, $\mathcal{F}_1(0)$ and  $ \mathcal{F}_2(0)$ are not invariant under rotations  since $\RE \Big((1-z)\big(-if_{1,0}(iz)\big)'\Big)=-2$ at $z=1/2-i/2$ and $(1-z^2)\big(-if_{2,0}(iz)\big)'=(1-z^2)^2/(1+z^2)^2$  maps $\mathbb{D}$ to the whole complex plane except the negative real axis.
In this section,  for certain  positive values  of $\lambda$, we give the sharp estimation of the generalized  Zalcman coefficient functional $\phi(f,n,m;\lambda)$ when $f \in \mathcal{F}_1(\beta)$ or $f \in \mathcal{F}_2(\beta)$.
\begin{theorem}\label{p4thm4}
If  $\mu   \geq  \max{\{ nm/\big((n+m-1)(1- \beta)\big), nm/(n+m-1)\}}$  and $f(z)=z+\sum_{n=2}^{\infty}a_n z^n \in \mathcal{F}_1(\beta)$ $(\beta <1)$, then for all $n,m=2, 3, \ldots$,
\begin{align*}
&\left|\mu a_n a_m - a_{n+m-1}\right|\leq  \mu B_n B_m -B_{n+m-1},
\end{align*}
where
\begin{align}
B_n=\dfrac{ 1+ 2 (n-1)(1-\beta)}{n} \quad (n \geq 2). \label{p4eq4.10}
\end{align}
The inequality is sharp.
\end{theorem}

\begin{proof}
Let $g(z):=(1-z) f'(z)$. Since $f \in \mathcal{F}_1(\beta)$, therefore $$\dfrac{g(z)-\beta}{1-\beta}= 1 + \sum_{n=1}^{\infty} c_n z^n \in \mathcal{P},$$ which gives
\begin{align}
c_n&= \dfrac{(n+1) a_{n+1} - n a_n}{1-\beta} \quad (n \geq 1)\notag  \intertext{and}
a_n&= \dfrac{1 + (1-\beta)(c_1 + c_2 + \cdots+ c_{n-1})}{n} \quad (n \geq 2). \label{p4eq4.11}
\end{align}
Since $|c_n| \leq 2$ $(n \geq 1)$, the equation  \eqref{p4eq4.11} gives
\begin{align}
|a_n| \leq B_n, \label{p4eq4.12}
\end{align} where $B_n$ is given by the equation \eqref{p4eq4.10}.
 For  fixed $n, m =2,3,\ldots$ and $\lambda \in \mathbb{R}$, choose the sequence  $\{z_k\}$ of complex numbers  by  $z_{n-2}= \lambda (1-\beta)a_m$, $ z_{n+m-3}=-(1-\beta)$, $z_k = 0$ for all $ k \neq n-2, n+m-3$.  Then Lemma \ref{p4lem2} yields
\begin{align}
&\left|\big(\lambda n  a_n a_m -(n+m-1) a_{n+m-1}\big)-\big(\lambda (n-1) a_{n-1} a_m -(n+m-2) a_{n+m-2}\big)\right|^2 \notag \\
&\leq |2 \lambda (1- \beta)a_m -  m a_m +(m-1) a_{m-1}|^2 -|m a_m - (m-1) a_{m-1}|^2 + 4 (1- \beta)^2 \notag \\
&= 4 \lambda(1-\beta)\big(\lambda (1-\beta)-m\big)  |a_m|^2 + 4(m-1) \lambda (1-\beta) \RE{ a_m \overline{a_{m-1}}} + 4(1-\beta)^2.\notag
\end{align}
If $\lambda \geq \max{\{m/(1-\beta),m\}}$, then by using  equation \eqref{p4eq4.12} in the above inequality,  we obtain
\begin{align}
&|\big(\lambda n  a_n a_m -(n+m-1) a_{n+m-1}\big)-\big(\lambda (n-1) a_{n-1} a_m -(n+m-2) a_{n+m-2}\big)|^2 \notag \\
&\leq 4(1-\beta)^2 \left(\lambda B_m -1\right)^2.
 \label{p4eq4.1}
\end{align}
For $\lambda \geq \max{\{m/(1-\beta),m\}}$, consider
\begin{align}
&|\lambda n a_n a_{m} - (n+m-1)a_{n+m-1}| \notag \\ &\leq  |\big(\lambda n a_n a_m - (n+m-1) a_{n+m-1}\big)-\big(\lambda(n-1) a_{n-1} a_m - (n+m-2) a_{n+m-2}\big)|+ \cdots  \notag  \\
&\quad{}+  |\big( 2\lambda   a_{2} a_m - (m+1) a_{m+1}\big)-\big(\lambda a_{1} a_m - m a_{m}\big)| + |\lambda a_{1} a_m - m a_{m}| \quad (a_1=1). \notag
\end{align}
By applying the inequality \eqref{p4eq4.1} and the bounds given by \eqref{p4eq4.12} in  the above inequality, we have
\begin{align}
(n+m-1)\left|\dfrac{\lambda n}{n+m-1} a_n a_{m} - a_{n+m-1}\right|
\leq  2(1-\beta)\left(\lambda B_m-1\right)(n-1) +(\lambda-m)B_m. \notag
\end{align}
On substituting $ \mu=\lambda n /(n+m-1)$ in  the above inequality and simplifying, we obtain
\begin{align}
\left|\mu a_n a_{m} - a_{n+m-1}\right|
\leq \mu B_n B_m - B_{n+m-1} \notag
\end{align}
where $\mu \geq \max{\{ nm/\big((n+m-1)(1- \beta)\big), nm/(n+m-1)\}}$ and $B_n$ is given by \eqref{p4eq4.10}. The result is sharp for the function $f_{1,\beta}$ given by \eqref{p4eq4.102}.
\end{proof}
For $\beta=0$ and $m=n$, we have the following.
\begin{corollary}
If $f(z)=z+\sum_{n=2}^{\infty}a_n z^n \in \mathcal{F}_1(0)$ and $\mu \geq n^2/(2n-1)$, then  $$|\mu a_n^2 - a_{2n-1}| \leq \dfrac{\mu (2n-1)^2}{n^2}+\dfrac{3-4n}{2n-1}.$$ The result is sharp.
\end{corollary}
\begin{theorem}\label{p4thm5}
If  $\mu   \geq  \max{\{ nm/\big((n+m-1)(1- \beta)\big), nm/(n+m-1)\}}$  and $f(z)=z+\sum_{n=2}^{\infty}a_n z^n \in \mathcal{F}_2(\beta)$ $(\beta <1)$, then for all $n,m=2, 3, \ldots$ except when both $n$ and $m$ are even,
\begin{align*}
&\left|\mu a_n a_m - a_{n+m-1}\right|\leq  \mu C_n C_m -C_{n+m-1}
\end{align*}
where,  for $n \geq 2$,
\begin{align}
C_n= \begin{cases}
\dfrac{1+(n-1)(1-\beta)}{n}, &\text{if $n$ is odd;}\\
1-\beta, &\text{if $n$ is even.} \label{p4eq4.13}
\end{cases}
\end{align}
The result is  sharp.
\end{theorem}

\begin{proof}
Let $g(z):=(1-z^2) f'(z)$. Since $f \in \mathcal{F}_2(\beta)$, therefore $(g(z)-\beta)/(1-\beta)= p(z)$ for some $p(z)= 1 + \sum_{n=1}^{\infty} c_n z^n \in \mathcal{P}.$
This gives
\begin{align}
c_n&= \dfrac{(n+1) a_{n+1} - (n-1) a_{n-1}}{1-\beta}, \notag \\
a_{2k}&= \dfrac{(1-\beta)(c_1 + c_3 + \cdots + c_{2k-1})}{2k}\label{p4eq4.01}\intertext{and}
a_{2k+1}&= \dfrac{1 + (1-\beta)( c_2 + c_4 + \cdots + c_{2k})}{2 k +1}. \label{p4eq5}
\end{align}
Since $|c_n| \leq 2$ $(n \geq 1)$, the equations \eqref{p4eq4.01} and \eqref{p4eq5} give
\begin{align}
|a_n| \leq C_n \label{p4eq4.14}
\end{align}
for all $n \geq 2,$ where $C_n$ is given by the equation \eqref{p4eq4.13}.
Define a function  $f_3: \mathbb{D} \to \mathbb{C}$ by
\begin{align}
 f_3(z)&=\dfrac{z(1-\beta)}{1-z}+ \dfrac{\beta}{2}\log{\dfrac{1+z}{1-z}}\notag \\
 &= z+ C_2 z^2 + C_3 z^3 + C_4 z^4 +  C_5 z^5 + \cdots. \label{p4eq4.15}
\end{align}
Clearly, the bounds given in \eqref{p4eq4.14} are  sharp for the function $f_3$.

For  fixed $n, m =2,3, \ldots$ and $\lambda \in \mathbb{R}$, choose the sequence  $\{z_k\}$ of complex numbers  by  $z_{n-2}= \lambda (1-\beta)a_m$, $ z_{n+m-3}=-(1-\beta)$, $z_k = 0$ for all $ k \neq n-2, n+m-3$.  Then Lemma \ref{p4lem2} yields
\begin{align}
&\left|\big(\lambda n  a_n a_m -(n+m-1) a_{n+m-1}\big)-\big(\lambda (n-2) a_{n-2} a_m -(n+m-3) a_{n+m-3}\big)\right|^2 \notag \\
&\leq |2 \lambda (1- \beta)a_m -  m a_m +(m-2) a_{m-2}|^2 -|m a_m - (m-2) a_{m-2}|^2 + 4 (1- \beta)^2 \notag \\
&= 4 \lambda(1-\beta)\big(\lambda (1-\beta)-m\big)  |a_m|^2 + 4(m-2) \lambda (1-\beta) \RE{ a_m \overline{a_{m-2}}} + 4(1-\beta)^2.\notag
\end{align}
If $\lambda \geq \max{\{m/(1-\beta),m\}}$, then an application of the equation \eqref{p4eq4.14} in the previous inequality gives
\begin{align}
&\left|\big(\lambda n  a_n a_m -(n+m-1) a_{n+m-1}\big)-\big(\lambda (n-2) a_{n-2} a_m -(n+m-3) a_{n+m-3}\big)\right| \notag \\
&\leq 2(1-\beta)\left( \lambda C_m-1\right). \label{p4eq5.1}
\end{align}
If $n=2$ and $m=2k+1$ $(k \geq 1)$, then
\begin{align}
|2 \lambda a_2 a_{m} - (m+1)a_{m+1}| &= |2 \lambda a_2 a_{2k+1} - (2k+2) a_{2k+2}|\notag\\
&= \left|\dfrac{\lambda (1-\beta)}{2k+1} c_1 \left(1 + (1-\beta)\sum_{j=1}^{k} c_{2j}\right) - (1-\beta) \sum_{j=1}^{k+1}c_{2j-1} \right|\notag \\
&= \left|\left(\dfrac{\lambda}{m}-1\right) (1-\beta)c_1 +  (1-\beta)\sum_{j=1}^{k} \left(\dfrac{\lambda(1-\beta)}{m}c_1 c_{2j} -c_{2 j+1}\right) \right|.\notag
\end{align}
For $\lambda \geq \max{\{m/(1-\beta),m\}}$, an application  of Lemma \ref{p4lem1} in the above equation gives
\begin{align}
|2\lambda a_2 a_{m} - (m+1)a_{m+1}| &\leq (1-\beta)\big(2 \lambda C_m -(1+m) \big)\label{p4eq5.03}
\end{align}
 where $C_m$ is given by the equation \eqref{p4eq4.13}.
If  $ n>2$ is even and $m$ is odd, then
\begin{align}
&|\lambda n a_n a_{m} -(n+m-1) a_{n+m-1}| \notag \\
&\leq  |\big(\lambda n a_n a_m - (n+m-1)a_{n+m-1}\big)-\big(\lambda (n-2) a_{n-2} a_m - (n+m-3) a_{n+m-3}\big)| + \cdots \notag  \\
&\quad{} +  |\big(4 \lambda a_{4} a_m -(m+3) a_{m+3}\big)-\big(2 \lambda a_{2} a_m - (m+1) a_{m+1}\big)| + |2 \lambda a_{2} a_m - (m+1)a_{m+1}|.\notag
\end{align}
For $\lambda \geq \max{\{m/(1-\beta),m\}}$, in view of \eqref{p4eq5.1} and \eqref{p4eq5.03}, we have
\begin{align*}
(n+m-1)\left|\dfrac{\lambda n }{n+m-1}a_n a_{m} -  a_{n+m-1}\right| \leq  (1-\beta)( \lambda  n C_m -n-m+1).
\end{align*}
On substituting $ \mu=\lambda n /(n+m-1)$ in the above inequality and simplifying, we obtain
\begin{align}
\left|\mu a_n a_{m} - a_{n+m-1}\right| \leq   (1-\beta)(\mu C_m-1) \notag
\end{align}
where $\mu \geq \max{\{ nm/\big((n+m-1)(1- \beta)\big), nm/(n+m-1)\}}.$
Next, we  consider the case when  $n$ is odd. In this case, we have
\begin{align}
&|\lambda n a_n a_{m} -(n+m-1) a_{n+m-1}| \notag \\
&\leq  |\big(\lambda n a_n a_m - (n+m-1)a_{n+m-1}\big)-\big(\lambda (n-2) a_{n-2} a_m - (n+m-3) a_{n+m-3}\big)| + \cdots \notag  \\
&\quad{} +  |\big(3 \lambda a_{3} a_m -(m+2) a_{m+2}\big)-\big( \lambda a_{1} a_m - m a_{m}\big)| + |\lambda a_{1} a_m - m a_{m}|\quad (a_1=1).\notag
\end{align}
For $\lambda \geq  \max{\{m/(1-\beta),m\}}$, use of the  equation \eqref{p4eq5.1} and  the bound of $|a_m|$ in the above inequality give
\begin{align*}
(n+m-1)\left|\dfrac{\lambda n}{n+m-1} a_n a_{m} - a_{n+m-1}\right|
\leq (n-1)(1-\beta)\left(\lambda C_m-1\right)+(\lambda -m)C_m
\end{align*}
where $C_m$ is given by  the equation \eqref{p4eq4.13}.
Substitution of $\mu=\lambda n/(n+m-1)$ in the previous inequality and simplification give
\begin{align*}
\left|\mu a_n a_{m} - a_{n+m-1}\right|
\leq  \mu C_n C_m -C_{n+m-1},
\end{align*}
where $\mu \geq \max{\{ nm/\big((n+m-1)(1- \beta)\big), nm/(n+m-1)\}}$.
The  sharpness  follows for the function $f_3$ given by the equation \eqref{p4eq4.15}.
\end{proof}
For $\beta=0$ and $m=n=2k+1$ $(k \geq 1)$, we have the following.
\begin{corollary}
If $f(z)=z+\sum_{n=2}^{\infty}a_n z^n \in \mathcal{F}_2(0)$ and $\mu \geq (2k+1)^2/(4k+1)$, then   $$|\mu a_{2 k+1}^2 - a_{4k+1}| \leq \mu-1  \quad  (k \geq 1). $$
The result is sharp.
\end{corollary}
\bibliography{ref4p}
\bibliographystyle{siam}

\end{document}